\documentclass[12pt]{amsart}
\usepackage{mathrsfs}
\usepackage{url}
\usepackage{color}
\usepackage[a4paper,asymmetric]{geometry}
\usepackage{latexsym}
\usepackage{amsthm}
\usepackage{amssymb}
\usepackage{amsfonts}
\usepackage{amsmath}

\newtheorem{theorem}{Theorem}[section]
\newtheorem{thm}[theorem]{Theorem}

\newtheorem{lem}[theorem]{Lemma}
\newtheorem{proposition}[theorem]{Proposition}

\newtheorem{corollary}[theorem]{Corollary}

\theoremstyle{definition}

\newtheorem{defn}[theorem]{Definition}

\theoremstyle{remark}

\newtheorem{rem}[theorem]{Remark}
\numberwithin{equation}{section}

 \DeclareMathAlphabet{\mathpzc}{OT1}{pzc}{m}{it}

 \newcommand{\M}{\mathcal{M}}

 \newcommand{\E}{\mathbb{E}}            % expectation
 \newcommand{\T}{\mathbb{T}}

 \newcommand{\Ll}{\langle}
 \newcommand{\Rr}{\rangle}

 \newcommand{\N}{\mathbb{N}}
 \newcommand{\R}{\mathbb{R}}
 \newcommand{\Z}{\mathbb{Z}}

 \newcommand{\FF}{\mathcal{F}}
 \newcommand{\PP}{\mathbb{P}}
 \newcommand{\mcl}{\mathcal}
 
 \newcommand{\Be}{\begin{equation}}
 \newcommand{\Ee}{\end{equation}}
 \newcommand{\Bs}{\begin{split}}
 \newcommand{\Es}{\end{split}}
  \newcommand{\Bes}{\begin{equation*}}
 \newcommand{\Ees}{\end{equation*}}
 \newcommand{\BT}{\begin{thm}}
 \newcommand{\ET}{\end{thm}}
 \newcommand{\Bp}{\begin{proof}}
 \newcommand{\Ep}{\end{proof}}
 \newcommand{\BL}{\begin{lem}}
 \newcommand{\EL}{\end{lem}}
 \newcommand{\BP}{\begin{proposition}}
 \newcommand{\EP}{\end{proposition}}
 \newcommand{\BC}{\begin{corollary}}
 \newcommand{\EC}{\end{corollary}}
 \newcommand{\BR}{\begin{rem}}
 \newcommand{\ER}{\end{rem}}
 \newcommand{\BD}{\begin{defn}}
 \newcommand{\ED}{\end{defn}}
 \newcommand{\BI}{\begin{itemize}}
 \newcommand{\EI}{\end{itemize}}
 \newcommand{\eqn}{equation}

  \newcommand{\dif}{{\rm d}}

\def\PP{\mathbb P}

\def\Om{{\Omega}}

\def\<{\left<}\def\>{\right>}
\def\({\left(}\def\){\right)}

\begin{document}
\title
[\ \ \ \ LDP for stochastic real Ginzburg-Landau equation]{Large deviation principle of occupation measure for stochastic real Ginzburg-Landau equation driven by $\alpha$-stable noises}

\author[R. Wang]{Ran Wang}
\address{School of Mathematical Sciences, University of Science and Technology of China, Hefei, China.} \email{wangran@ustc.edu.cn}

 \author[J. Xiong]{Jie Xiong }
\address{Department of Mathematics, Faculty of Science and Technology,  University of  Macau, Taipa, Macau.}
\email{jiexiong@umac.mo}

\author[L. Xu]{Lihu Xu}
\address{Department of Mathematics, Faculty of Science and Technology,  University of  Macau, Taipa, Macau.}
\email{lihuxu@umac.mo}

%\thanks{The first author
%was supported by the M.I.U.R. research project Prin 2008 ``Deterministic and stochastic methods in the study of evolution problems''.
%The third author gratefully acknowledges the support by
%Junior program \emph{Stochastics} of Hausdorff Research Institute for
%Mathematics. His research is partially supported by the European Research Council under the European Union's
%Seventh Framework Programme (FP7/2007-2013) / ERC grant agreement nr. 258237.}
%\subjclass[2000]{}
%\keywords{}
%\date{}
%%% ----------------------------------------------------------------------
\maketitle
\begin{minipage}{140mm}
\begin{center}
{\bf Abstract}
\end{center}

  We establish a large deviation principle for the occupation measure of the stochastic real Ginzburg-Landau equation driven by  $\alpha$-stable noises.  The proof is based on a hyper-exponential recurrence criterion. Our result indicates a phenomenon that
  strong dissipation beats heavy tailed noises to produce a large deviation, it seems to us that this phenomenon has not been reported in the known literatures.
\end{minipage}

\vspace{4mm}



%%% ----------------------------------------------------------------------

\section{Introduction}
 Consider the stochastic real Ginzburg-Landau equation driven by $\alpha$-stable noises on torus $\mathbb T:=\mathbb R/\mathbb Z$ as follows:
\begin{\eqn} \label{e:MaiSPDE}
\dif X-\partial_{\xi}^2X\dif t-(X-X^3)\dif t= \dif L_t,
\end{\eqn}
where $X:[0,+\infty)\times \mathbb T\times\Om\rightarrow\mathbb R$ and $L_t$ is an   $\alpha$-stable noise.
  Eq. \eqref{e:MaiSPDE} admits a unique mild solution $X$ in the c\`adl\`ag space almost surely.
When $\alpha \in (3/2,2)$, the  Markov process  $X$ is strong Feller and irreducible, and it has a unique invariant measure $\pi$. See \cite{Xu13, WXX} or Section 2  below for  details.
By the uniqueness \cite{DPZ96}, $\pi$ is ergodic in the sense that
$$
\lim_{T\rightarrow \infty}\frac1T\int_0^T \Psi(X_t)\dif t=\int \Psi \dif \pi\ \ \ \ \mathbb P {\text-a.s.}
$$
for all initial state $X_0$ and all continuous and bounded functions $\Psi$.

Let $\mathcal L_t$ be the occupation measure of the system \eqref{e:MaiSPDE}  given by
\begin{equation}\label{e:occupation}
\mathcal L_t(A):=\frac1t\int_0^t\delta_{X_s}(A)\dif s \ \ \  \ \text{ for any measurable set } A,
\end{equation}
where $\delta_a$ is the Dirac measure at $a$.
In our previous paper \cite{WXX}, we prove that the system converges to its invariant measure $\mu$ with exponential rate under a topology stronger than total variation, and the occupation measure $\mathcal L_t$ obeys the moderate deviation principle by constructing some Lyapunov test functions.

\vskip0.3cm
In this paper, we further study  the  Large Deviation Principle (LDP) for the occupation measure $\mathcal L_t$.  The LDP for empirical measures is one of the strongest ergodicity results
for the long time behavior of Markov processes. It has been one of the classical research
topics in probability since the pioneering work of Donsker and Varadhan \cite{DV}. Refer to the books \cite{DZ,DS, FW}.

Based on the hyper-exponential recurrence criterion developed by Wu \cite{Wu01}, we  prove that the occupation measure $\mathcal L_t$ obeys an LDP under $\tau$-topology.  As a consequence, we also obtain the exact rate of exponential ergodicity obtained in \cite{Xu13}.

\vskip0.3cm
For  stochastic partial differential equations,
the problems of  large deviations   have been intensively studied in recent years. Most of them, however, are concentrated on the small noise large deviation principles of Freidlin-Wentzell's type, which provide estimates for the probability that stochastic systems converge to their deterministic part as noises tend to zero,  see \cite{CR, Fre, KX, Sowers, Xu-Zhang} and the references therein.  But there are only very few papers on  the  large deviations of Donsker-Varadhan's  tpye for  large time, which  estimate   the probability  of the occupation measures'  deviation from invariant measure. Gourcy \cite{Gou1, Gou2}  established the LDP for occupation measures of stochastic Burgers and Navier-Stokes equations by the means of the hyper-exponential recurrence. Jak\u{s}i\`c et al. \cite{JNPS1} established the LDP for occupation measures of SPDE with smooth random perturbations by Kifer's large deviation criterion \cite{Kif}. Jak\u{s}i\`c et al. \cite{JNPS2} also gave the large deviations estimates for dissipative PDEs with rough noise by the hyper-exponential recurrence criterion. Our result indicates a phenomenon that
  strong dissipation beats heavy tailed noises to produce a large deviation. It seems to us that this phenomenon has not been reported in the known literatures, we shall concentrate on studying it in future research.

\vskip0.3cm

Finally we recall some literatures on the study of invariant measures and the long time behavior of stochastic systems driven by $\alpha$-stable type noises. \cite{PSXZ11, PXZ10} studied the exponential mixing for a family of semi-linear SPDEs with Lipschitz nonlinearity, while \cite{DXZ11} obtained the existence of invariant measures for 2D stochastic Navier-Stokes equations forced by $\alpha$-stable noises with $\alpha\in (1,2)$. \cite{Xu14} proved the exponential mixing for a family of 2D SDEs forced by degenerate $\alpha$-stable noises.
For the long term behaviour about stochastic system drive by L\'evy noises,
we refer to \cite{Do08, DXi11, DXZ09,  FuXi09, Mas} and the literatures therein.

\vskip0.3cm

The paper is organized as follows. In Section 2, we  give a brief review of some known results about  the   stochastic  Ginzburg-Landau equations, and present the main result of this paper. In Section 3, we  recall some general facts about LDP for strong Feller and topologically irreducible   Markov processes. In the last  section, we  will prove the main result.

Throughout this paper, $C_p$ is a positive constant depending on some parameter $p$, and $C$ is a constant depending on no  specific parameter (except $\alpha, \beta$), whose value  may be different from line to line by convention.

\section{The model and the results}

Let $\T= \R/\Z$ be equipped with the usual Riemannian metric, and let $\dif \xi$
denote the Lebesgue measure on $\T$. For any $p\ge1$, let
$$
L^p(\T;\R):=\left\{x: \T\rightarrow\R; \|x\|_{L^p}:=\left(\int_\T |x(\xi)|^p \dif\xi\right)^{\frac1p}<\infty\right\}.
 $$
Denote
$$H:=\bigg\{x\in L^2(\T; \R); \int_\T x(\xi) \dif\xi =0\bigg\},$$
it is a separable real Hilbert space with inner product
$$\Ll x,y \Rr_H:=\int_\T x(\xi)y(\xi) \dif\xi,\ \ \ \ \ \forall \ x, y \in H.$$
For any $x\in H$, let
$$
\|x\|_H:=\|x\|_{L^2}=\left(\langle x,x\rangle_H\right)^{\frac12}.
$$
Let $\Z_*:=\Z \setminus \{0\}$. It is well known that
$$\left\{e_k; e_k=e^{i2 \pi k\xi}, \ k \in \Z_*\right\} $$
is an orthonormal basis of $H$. For each $x \in H$,
it can be represented by  Fourier series
$$x=\sum_{k \in \Z_*} x_k e_k  \ \ \ \ {\rm with} \ \ \ x_k \in \mathbb C, \ x_{-k}=\overline{x_k}.$$

Let $\Delta$ be the Laplace operator on $H$. It is well known that
$D(\Delta)=H^{2,2}(\T) \cap H$. In our setting, $\Delta$ can be determined by the following
relations: for all $k \in \Z_*$,
$$\Delta e_k=-\gamma_k e_k\ \ \ \ {\rm with} \ \ \gamma_k=4 \pi^2 |k|^2,$$
with
$$H^{2,2}(\T) \cap H=\left\{x \in H; \ x=\sum_{k \in \Z_*} x_k e_k, \ \sum_{k \in \Z_*} |\gamma_k|^{2} |x_k|^2<\infty\right\}.$$
Denote
$$A=-\Delta, \ \ \ \ D(A)=H^{2,2}(\T) \cap H.$$
Define the operator $A^{\sigma}$ with $\sigma \ge 0$ by
$$A^\sigma x=\sum_{k \in \Z_*} \gamma_k^{\sigma} x_ke_k, \ \ \ \ \ \ x \in D(A^\sigma),$$
where $\{x_k\}_{k \in \Z_*}$ are the Fourier coefficients of $x$, and
$$D(A^\sigma):=\left\{x \in H: \ x=\sum_{k \in \Z_*} x_k e_k, \sum_{k \in \Z_*} |\gamma_k|^{2 \sigma} |x_k|^2<\infty\right\}.$$

Given $x \in D(A^\sigma)$, its norm is
$$\|A^\sigma x\|_H:=\left(\sum_{k \in \Z_*} |\gamma_k|^{2 \sigma} |x_k|^2\right)^{1/2}.$$
For $\sigma>0$, let
$$
H_{\sigma}:=D(A^\sigma), \ \  \ \  \|x\|_{H_{\sigma}}:=\|A^\sigma x\|_H.$$
Then, $H_{\sigma}$ is densely and compactly embedded in $H$.  Particularly, let
$$V:=D(A^{1/2}).$$

We shall study  $1$D stochastic Ginzburg-Landau equation on $\T$ as the following
\begin{equation} \label{e:XEqn}
\begin{cases}
\dif X_t + A X_t\dif t= N(X_t) \dif t + \dif L_t, \\
X_0=x,
\end{cases}
\end{equation}
where
\begin{itemize}
\item[(i)] the nonlinear term $N$ is defined by
\begin{equation*} \label{e:NonlinearB}
N(u)= u-u^3 , \ \ \ \ \ u \in H.
\end{equation*}
\item[(ii)] $L_t=\sum_{k \in \Z_*} \beta_k l_k(t) e_k$ is an $\alpha$-stable process on $H$ with
$\{l_k(t)\}_{k\in \Z_*}$ being i.i.d. 1-dimensional symmetric $\alpha$-stable process sequence with $\alpha>1$, see \cite{sato}. Moreover, we assume that there exist some $C_1, C_2>0$ so that $C_1 \gamma_k^{-\beta} \le |\beta_k| \le C_2 \gamma_k^{-\beta}$ with $\beta>\frac 12+\frac 1{2\alpha}$.
\end{itemize}

We shall   use the following inequalities (see \cite{Xu13}):
%\Be \label{e:PoiInq}
%\|A^{\sigma_1} x\|_H \le  C_{\sigma_1, \sigma_2} \|A^{\sigma_2} x\|_H, \ \ \ \ \ \ \forall \ \sigma_1 \le \sigma_2 \ \ \forall x \in H;
%\Ee
\Be \label{e:eAEst}
\|A^{\sigma} e^{-At}\|_H \le C_\sigma t^{-\sigma}, \ \ \ \ \ \forall \ \sigma>0 \ \ \ \forall \ t>0;
\Ee
%\Be \label{e:NInnPro}
%\langle x, -N(x)\rangle_H \le \frac 14, \ \ \ \ \forall \ x \in H;
%\Ee
%\Be \label{e:NVEst}
%\|N(x)\|_V \leq C (\|x\|_V+\|x\|^3_V), \ \ \ \ \forall \ x \in V;
%\Ee
%\Be\label{e: NAH}
%\|AN(x)\|_H\le C(1+\|x\|_V^2)(1+\|Ax\|_H^2);
%\Ee
\Be \label{e:L4}
\|x\|^4_{L^4}\le \|x\|_V^2\cdot\|x\|_H^2\le \|x\|_V^4, \ \ \ \ \  \forall  \ x\in V;
\Ee
\Be \label{e:3H}
\|x^3\|_{H}\le C\|A^{\frac14}x\|_H^2\cdot\|x\|_H \le C\|x\|_V\|x\|_H^2 , \ \ \ \ \  \forall  \ x\in V.
\Ee

\vskip0.3cm

\BD We say that a predictable $H$-valued stochastic process $X=(X_t^x)$ is a mild solution to Eq. \eqref{e:XEqn} if, for any $t\ge0, x\in H$, it holds ($\mathbb P$-a.s.):
\begin{equation}\label{e: mild solution}
X^x_t(\omega)=e^{-At} x+\int_0^t e^{-A(t-s)} N(X_s^x(\omega))\dif s+\int_0^t e^{-A(t-s)} \dif L_s(\omega).
\end{equation}
\ED

The following  properties for the solutions   can be found in \cite{Xu13, WXX}.
\begin{thm}[\cite{Xu13, WXX}] \label{Thm Xu 13} Assume that $\alpha \in (3/2,2)$ and $\frac 12+\frac{1}{2\alpha}<\beta<\frac 32-\frac{1}{\alpha}$, the following statements hold:
\begin{enumerate}
\item    For every $x \in H$ and $\omega \in \Omega$ a.s.,
Eq. \eqref{e:XEqn} admits a unique mild solution $X=(X^x_t)_{t\ge0,x\in H} \in D([0,\infty);H) \cap D((0,\infty);V)$.

\item  $X$ is a Markov process, which is strong Feller and irreducible in $H$, and
 $X$ admits a unique invariant measure $\pi$.
  %Furthermore,  there exist some positive constants $M>1, \rho\in(0,1), \theta>0$ satisfying that  $\int \Psi\dif \pi<+\infty$, where $\Psi(x):= (M+\|x\|_H^2)^{1/2}$,
  %and $\pi$ is exponentially ergodic in the sense that
%\begin{equation*}
%\sup_{|f|\le \Psi}\left|P_tf(x)-\int f\dif \pi \right|\le \theta \Psi(x)\cdot \rho^t\ \ \  \forall x\in H,  t\ge0.
%\end{equation*}
\end{enumerate}
\end{thm}

Let $C_b(H)$ (resp. $b\mathcal B(H)$) be the space of all bounded continuous  (resp. measurable) functions on $H$.
Let   $\mathcal M_1(H)$ be the space of  probability measures on $H$ equipped with the Borel $\sigma$-field $\mathcal B (H)$.
On  $\mathcal M_1(H)$, let  $\sigma(\mathcal M_1(H), b\mathcal B(H))$ be the $\tau$-topology  of converence against measurable and bounded functions  which is much stronger than the usual weak convergence topology $\sigma(\mathcal M_1(H), C_b(H))$, see \cite[Section 6.2]{DZ}.
\vskip0.3cm

Now, we are at the position to state our main result. Recall $\mathcal L_T$ defined by \eqref{e:occupation}.
\begin{theorem}\label{thm main} Assume that $\alpha \in (3/2,2)$ and  $\frac 12+\frac{1}{2\alpha}<\beta<\frac 32-\frac{1}{\alpha}$.
Then the family $\PP_{\nu}(\mathcal L_T\in \cdot)$ as $T\rightarrow +\infty$ satisfies the large deviation principle with respect to the $\tau$-topology, with speed $T$ and rate function $J$ defined by \eqref{rate func} below, uniformly for any initial measure $\nu$ in $\mathcal M_1(H)$.  More precisely, the following three properties hold:
\begin{itemize}
  \item[(a1)]  for any $a\ge0$, $\{\mu\in \mathcal M_1(H); J(\mu)\le a \}$ is compact in  $(\mathcal M_1(H),\tau)$;
  \item[(a2)] (the lower bound) for any  open set $G$ in $(\mathcal M_1(H),\tau)$,
   $$
   \liminf_{T\rightarrow \infty}\frac1T\log\inf_{\nu\in\mathcal M_1(H)}\mathbb P_{\nu}(\mathcal L_T\in G)\ge -\inf_G J;
   $$
  \item[(a3)](the upper bound) for any  closed set $F$ in  $(\mathcal M_1(H),\tau)$,
   $$
   \limsup_{T\rightarrow \infty}\frac1T\log\sup_{\nu\in\mathcal M_1(H)}\mathbb P_{\nu}(\mathcal L_T\in F)\le -\inf_F J.
   $$
\end{itemize}
\end{theorem}

\begin{rem} For every $f:H\rightarrow \R$ measurable and bounded, as $\nu\rightarrow\int_{H}f\dif \nu$ is continuous w.r.t. the $\tau$-topology, then by the contraction principle (\cite[Theorem 4.2.1]{DZ}), $$\mathbb P_{\nu}\left(\frac{1}{T}\int_0^T f(X_s)\dif s\in \cdot\right)$$
satisfies the LDP on $\R$ uniformly for any initial measure $\nu$ in $\mathcal M_1(H)$,  with the rate function given by
$$
J^f(r)=\inf\left\{J(\nu)<+\infty|\nu\in\mathcal M_1(H)\  \text{and } \int f\dif\nu=r \right\},\ \ \forall r\in\R.
$$
\end{rem}

\section{General results about large deviations}

In this section, we recall some general results on the Large Deviation Principle for strong Feller and irreducible   Markov processes. We follow \cite{Wu01}.

\vskip0.3cm
\subsection{Notations and entropy of Donsker-Varadhan}\ \

Let $E$ be a Polish metric space. Consider a general $E$-valued c\`adl\`ag  Markov process
$$
\left(\Omega, \{\mathcal F_t\}_{t\ge0}, \mathcal F, \{X_t(\omega)\}_{t\ge0}, \{\mathbb P_x\}_{x\in E}\right),
$$
where
\begin{itemize}
  \item  $\Omega=D( [0,+\infty); E)$, which is the space of the c\`adl\`ag functions from $[0,+\infty)$ to $E$ equipped with the  Skorokhod topology; for any $\omega\in \Omega$, $X_t(\omega)=\omega(t)$;
  \item $\FF_t^0=\sigma\{X_s; 0\le s\le t\}$ for any $t\ge 0$ (nature filtration);
  \item $\FF=\sigma\{X_t; t\ge0\}$ and $\PP_x(X_0=x)=1$.
\end{itemize}
 Hence, $\PP_{x}$ is the law of the Markov process with initial state $x\in E$.  For any initial measure $\nu$ on $E$, let $\PP_{\nu}(\dif \omega):=\int_E \PP_x(\dif \omega)\nu(\dif x)$. Its transition probability is denoted by $\{P_t(x, dy)\}_{t\ge0}$.

For all $f\in b\mathcal B(E)$, define
$$
P_tf(x)=\int_E P_t(x, \dif y)f(y)  \ \ \ \text{for all } t\ge0, x\in E.
$$

For any $t>0$, $P_t$ is said to be {\it strong Feller} if $P_t\varphi\in C_b(E)$ for any $\varphi\in b\mathcal B(E)$; $P_t$ is {\it irreducible} in $E$ if $P_t1_O(x)>0$ for any $x\in E$ and any non-empty open subset $O$ of $E$. $\{P_t\}_{t\ge0}$ is  {\it accessible }to $x \in E$,  if the resolvent $\{\mcl R_{\lambda}\}_{\lambda>0}$ satisfies
$$\mcl R_{\lambda}(y, \mcl U):= \int_0^\infty e^{-\lambda t}P_t(y, \mcl U) \dif t>0 , \ \ \forall \lambda>0$$
for all $y \in E$ and all neighborhoods $\mcl U$ of $x$. Notice that the accessibility of $\{P_t\}_{t\ge0}$ to any $x\in E$ is   the so called {\it topological transitivity} in Wu \cite{Wu01}.

\vskip0,3cm

The empirical measure of level-$3$ (or process level) is given by
$$
R_t:=\frac1t\int_0^t \delta_{\theta_s X}\dif s
$$
where $(\theta_sX)_t=X_{s+t}$ for all $t, s\ge0$ are the shifts on $\Omega$. Thus, $R_t$ is a random element of $\mathcal M_1(\Omega)$, the space of all probability measures on $\Omega$.

The level-$3$ entropy functional of Donsker-Varadhan $H:\mathcal M_1(\Omega)\rightarrow [0,+\infty]$ is defined by
\begin{equation*} \label{e: DV}
H(Q):=\begin{cases}
 \mathbb E^{\bar Q}h_{\mathcal F_1^0}(\bar Q_{w(-\infty,0]};\mathbb P_{w(0)}) & \text{if } Q\in \mathcal M_1^s(\Omega);  \\
+\infty & \text{otherwise},
\end{cases}
\end{equation*}
where
 \begin{itemize}
   \item $\mathcal M_1^s(\Omega)$ is the subspace of $\mathcal M_1(\Omega)$, whose  elements are moreover stationary;
   \item $\bar Q$ is the unique stationary extension of $Q\in \mathcal M_1^s(\Omega)$ to $\bar \Omega:=D(\mathbb R; E)$; $\mathcal F_t^s=\sigma\{X(u); s\le u\le t\},\forall  s,t\in\R, s\le t$;
   \item $\bar Q_{w(-\infty,t]}$ is the regular conditional distribution of $\bar Q$ knowing $\mathcal F_t^{-\infty}$;
   \item $h_{\mathcal G}(\nu;\mu)$ is the usual relative entropy or Kullback information of $\nu$ with respect to $\mu$ restricted to the $\sigma$-field $\mathcal G$, given by
\begin{equation*}
h_{\mathcal G}(\nu;\mu):=\begin{cases}
  \int\frac{\dif\nu}{\dif\mu}|_{\mathcal G} \log\left(\frac{\dif\nu}{\dif\mu}|_{\mathcal G}\right) \dif\mu & \text{ if }  \nu\ll \mu \text{ on } \ \mathcal G;  \\
+\infty & \text{otherwise}.
\end{cases}
\end{equation*}
 \end{itemize}

The level-$2$ entropy functional $J: \mathcal M_1(E)\rightarrow [0, \infty]$ which governs the LDP in our main result is
\begin{equation}\label{rate func}
J(\mu)=\inf\{H(Q)| Q\in \mathcal M_1^s(\Omega) \  \ \text{and } Q_0=\mu\}, \ \ \ \ \forall \mu\in \mathcal M_1(E),
\end{equation}
where $Q_0(\cdot)=Q(X_0\in \cdot)$ is the marginal law at $t=0$.

\subsection{The hyper-exponential recurrence criterion}
   Recall the following hyper\linebreak[4] -exponential recurrence criterion for LDP established by Wu \cite[Theorem 2.1]{Wu01}, also see Gourcy \cite[Theorem 3.2]{Gou2}.

\vskip0.3cm
For  any measurable set $K\in E$, let
\begin{equation}\label{stopping time}
\tau_K=\inf\{t\ge0 \ \text{ s.t.}\   X_t\in K\},\ \ \ \tau_K^{(1)}=\inf\{t\ge1\  \text{ s.t.}\ X_t\in K\}.
\end{equation}

\begin{theorem}\cite{Wu01}\label{thm Wu}
Let $\mathcal A\subset \mathcal M_1(E)$ and assume that
\begin{equation*}\label{condition 1}
\{P_t\}_{t\ge0} \text{ is strong Feller and topologically irreducible on  } E.
\end{equation*}
If for any $\lambda>0$, there exists some compact set $K\subset \subset E$, such that
\begin{equation}\label{condition 2}
\sup_{\nu\in\mathcal A}\E^{\nu}e^{\lambda\tau_K}<\infty, \ \  \text{and} \ \ \
\sup_{x\in K}\E^{x}e^{\lambda\tau_K^{(1)}}<\infty.
\end{equation}
 Then the family $\mathbb P_{\nu}(\mathcal L_t\in\cdot)$ satisfies the LDP on $\mathcal M_1(E)$ w.r.t. the $\tau$-topology with the rate function $J$ defined by \eqref{rate func}, and uniformly for initial measures $\nu$ in the subset $\mathcal A$. More precisely, the following three properties hold:
\begin{itemize}
  \item[(a1)] for any $a\ge0$, $\{\mu\in \mathcal M_1(E); J(\mu)\le a \}$ is compact in  $(\mathcal M_1(E),\tau)$;
  \item[(a2)] (the lower bound) for any open set $G$ in $(\mathcal M_1(E), \tau)$,
   $$
   \liminf_{T\rightarrow \infty}\frac1T\log\inf_{\nu\in\mathcal A}\mathbb P_{\nu}(\mathcal L_T\in G)\ge -\inf_G J;
   $$
  \item[(a3)](the upper bound) for any  closed set $F$ in $(\mathcal M_1(E), \tau)$,
   $$
   \limsup_{T\rightarrow \infty}\frac1T\log\sup_{\nu\in\mathcal A}\mathbb P_{\nu}(\mathcal L_T\in F)\le -\inf_F J.
   $$
\end{itemize}

\end{theorem}

\section{The proof of the main theorem}
In this section, we shall prove  Theorem \ref{thm main} according to Theorem \ref{thm Wu}.

\begin{proof}[Proof of Theorem \ref{thm main}]
\ Let $\{X_t\}_{t\ge0}$ be the solution to Eq. \eqref{e:XEqn} with initial value $x\in H$. By Theorem \ref{Thm Xu 13}, we know that $X$ is strong Feller and irreducible in $H$.  According to Theorem \ref{thm Wu}, to prove Theorem \ref{thm main}, we need prove that the hyper-exponential recurrence condition \ref{condition 2} is fulfilled. The verification of this condition  will be given by Theorem  \ref{thm hyper-exp} below.
\end{proof}

\subsection{Some estimates} \label{section recur}
In this part, we will give some prior estimates, which are necessary for verifying the hyper-exponential recurrence condition \eqref{condition 2}.

\vskip0.3cm
Let $Z_t$ be the following  Ornstein-Uhlenbeck process:
\Be\label{e:OUAlp}
\dif Z_t+A Z_t \dif t= \dif L_t, \ \ \ Z_0=0,
\Ee
where $L_t=\sum_{k \in \Z_*} \beta_k l_k(t) e_k$ is the $\alpha$-stable process defined in Eq. \eqref{e:XEqn}. It is well known that
$$
Z_t=\int_0^t e^{-A(t-s)} \dif L_s=\sum_{k \in \Z_{*}} z_{k}(t) e_k,
$$
where $$z_{k}(t)=\int_0^t e^{-\gamma_k(t-s)}
\beta_k \dif l_k(s).$$
\vskip0.3cm
The following maximal inequality can be found  in \cite[Lemma 3.1]{Xu13}.
\begin{lem} \label{l:ZEst}  For  any $T>0, 0 \leq \theta<\beta-\frac 1{2 \alpha}$ and any $0<p<\alpha$, we have
$$
\E \left[\sup_{0 \leq t \le T}\|A^\theta Z_t\|^p_{H}\right] \le CT^{p/\alpha},
$$
where $C$ depends on $\alpha,\theta, \beta, p$.
\end{lem}

\vskip0.3cm

Let $Y_t:=X_t-Z_t$. Then $Y_t$ satisfies the following equation:
\begin{equation}\label{e:Y}
\dif Y_t+A Y_t\dif t=N(Y_t+Z_t)\dif t, \ \ \ \ Y_0=x.
\end{equation}

\begin{lem}\label{l:Y}  For all $T>0$, we have
 \begin{equation}  \label{e:YtEstZt}
\sup_{t\in [T/2,T]}\|Y_t\|_H\le C(T)\left(1+\sup_{0\le t\le T}\|Z_t\|_{V} \right),
\end{equation}
 where the constant $C(T)$     does not depend on the initial value $Y_0=x$.
\end{lem}

\begin{proof}

By the chain rule, we obtain that
\begin{equation}\label{e:Y1}
\frac{\dif \|Y_t\|_H^2}{\dif t}+2\|Y_t\|_V^2=2\langle Y_t, N(Y_t+Z_t) \rangle.
\end{equation}
Using the following Young inequalities: for any $y,z\in L^4(\T; \R)$ and $C_1>0$, there exists $C_2>0$ satisfying that
\begin{equation*}
\begin{split}
& |\Ll  y, z^3 \Rr_H|=\left|\int_{\mathbb T} y(\xi) z^3 (\xi) \dif \xi\right| \le \frac{\int_{\mathbb T} y^4(\xi) \dif \xi}{C_1}+ C_2\int_{\mathbb T} z^{4}(\xi) \dif \xi, \\
& |\Ll  y^2, z^2 \Rr_H|=\left|\int_{\mathbb T} y^2(\xi) z^2 (\xi) \dif \xi\right| \le \frac{\int_{\mathbb T} y^4(\xi) \dif \xi}{C_1}+C_2\int_{\mathbb T} z^4(\xi) \dif \xi, \\
& |\Ll  y^3, z  \Rr_H|=\left|\int_{\mathbb T} y^3(\xi) z(\xi) \dif \xi\right| \le \frac{ \int_{\mathbb T} y^4(\xi) \dif \xi}{C_1}+ C_2\int_{\mathbb T} z^4(\xi) \dif \xi,
\end{split}
\end{equation*}
and using    H\"older inequality and the elementary inequality $2\sqrt a\le a/b+b$ for all $a,b>0$, we  obtain that there exists a constant $C\ge1$ satisfying that
$$
2\langle Y_t, N(Y_t+Z_t) \rangle\le -\|Y_t\|_{L^4}^4+C(1+\|Z_t\|_{L^4}^4).
$$
This  inequality, together with  Eq. \eqref{e:L4},  Eq. \eqref{e:Y1} and   H\"older inequality, implies that
\begin{equation}\label{e:Y2}
\frac{\dif \|Y_t\|_H^2}{\dif t}+2\|Y_t\|_V^2\le -\|Y_t\|_{H}^4+C\left(1+\|Z_t\|_V^4\right).
\end{equation}

For any $t\ge0$, denote $$h(t):= \|Y_t\|_H^2,\ \  \ K_T:=\sup_{0\le t\le T}\sqrt{C(1+\|Z_t\|_V^4)}\ge 1.$$
 By Eq. \eqref{e:Y2}, we have
$$
\frac{\dif h(t)}{\dif t} \le -h^2(t)+K_T^2, \ \ \ \forall t\in[0,T],
$$
with the initial value $h(0)=\|x\|_H^2\ge0$.

By the comparison theorem (e.g., the deterministic case of   \cite[Chapter VI, Theorem 1.1]{IW}),  we obtain that
\begin{equation}\label{e:h g}
h(t)\le g(t), \ \ \ \ \ \ \forall t\in[0,T],
\end{equation}
where the function $g$ solves the following equaiton
\begin{equation}\label{e:g}
\frac{\dif g(t)}{\dif t}= -g^2(t)+K_T^2, \ \ \ \forall t\in[0,T],
\end{equation}
with the initial value $g(0)=h(0)$. The solution of  Eq. \eqref{e:g} is
\begin{equation*}
g(t)=K_T+2K_T\left( \frac{g(0)+K_T}{g(0)-K_T}e^{2K_T t}-1\right)^{-1}, \ \  \ \forall t\in[0,T],
\end{equation*}
where  it is understood that  $g(t)\equiv K_T$ when  $g(0)=K_T$.  It is easy to show that
  \begin{itemize}
   \item[(1)]if  $g(0)\in [0,K_T]$, we have
   $$
   g(t)\le K_T, \ \ \ \forall t\in[0,T];
   $$
       \item[(2)] if $g(0)>K_T$, noticing that $K_T\ge 1$, we have that for all $t\in[T/2,T]$,
          \begin{align*}
g(t)\le& K_T+2K_T\left(  e^{2K_T t}-1\right)^{-1}\\
\le &  K_T\left(1+2(   e^{T}-1)^{-1}\right).
       \end{align*}
 \end{itemize}
 Therefore, for any initial value $g(0)$, we have
 $$
 g(t)\le K_T\left(1+2(   e^{T}-1)^{-1}\right), \ \ \ \forall t\in[T/2,T].
  $$
  This inequlity, together with Eq. \eqref{e:h g} and the definition of $K_T$, immediately implies the required estimate \eqref{e:YtEstZt}.

         The proof is complete.
\end{proof}

\vskip0.3cm

\begin{lem}\label{l:Y2}
For all $T>0$, $\delta \in (0,1/2)$ and  $p \in (0,\alpha/4)$, we have
\Be
\E_x\left[\|Y_{T}\|^p_{H_\delta} \right]\le C_{T,\delta,p}
\Ee
where the constant $C_{T,\delta, p}$  does not depend on the initial  value $Y_0=x$.
\end{lem}
\begin{proof} Since
$$Y_{T}=e^{-AT/2} Y_{T/2}+\int_{T/2}^T e^{-A(T-s)} N(Y_s+Z_s) \dif s,$$
for any $\delta \in (0,1/2)$, by the inequalities \eqref{e:eAEst}-\eqref{e:3H} and Lemma \ref{l:Y}, there exists a constant $C=C_{T,\delta}$ (whose value may be different from line to line by convention) satisfied  that
\begin{align*}
\|Y_{T}\|_{H_\delta}
\le & C \|Y_{T/2}\|_H+C\int_{T/2}^T (T-s)^{-\delta} \|N(Y_s+Z_s)\|_H \dif s\notag \\
  \le& C \|Y_{T/2}\|_H+C\int_{T/2}^T (T-s)^{-\delta} (\|Y_s\|_H+\|Z_s\|_H+ \|Y_s^3\|_H+\|Z_s^3\|_H) \dif s \notag \\
 \le & C \|Y_{T/2}\|_H+C\int_{T/2}^T (T-s)^{-\delta} (\|Y_s\|_H+\|Z_s\|_V+\|Y_s\|_{V} \|Y_s\|^2_H+\|Z_s\|^3_V) \dif s \notag\\
  \le & C \left(1+\sup_{0 \le t \le T} \|Z_t\|^3_V\right)+C\int_{T/2}^T (T-s)^{-\delta} \|Y_s\|_{V} \|Y_s\|^2_H \dif s.
\end{align*}
Next, we estimate  the last term in above inequality: by Eq. \eqref{e:Y2} and Lemma \ref{l:Y} again, we have
\begin{align*}
&\int_{T/2}^T (T-s)^{-\delta} \|Y_s\|_{V} \|Y_s\|^2_H \dif s\\
 \le &C\left(1+\sup_{0 \le t \le T}\|Z_t\|^2_V\right)\int_{T/2}^T (T-s)^{-\delta} \|Y_s\|_{V} \dif s \\
\le& C\left(1+\sup_{0 \le t \le T}\|Z_t\|^2_V\right)\left(\int_{T/2}^T (T-s)^{-2\delta} \dif s\right)^{\frac 12} \left(\int_{T/2}^T\|Y_s\|^2_{V} \dif s\right)^{\frac 12}  \\
\le& C\left(1+\sup_{0 \le t \le T}\|Z_t\|^2_V\right) \left(\|Y_{T/2}\|^2_H+\int_{T/2}^T (1+\|Z_s\|^4_V)\dif s\right)^{\frac 12} \\
\le &  C\left(1 +\sup_{0 \le t \le T} \|Z_t\|^4_V\right).
\end{align*}
Hence, by Lemma \ref{l:ZEst}, we obtain that for any $p \in (0,\alpha/4)$,
$$
\E_x\left[\|Y_{T}\|^p_{H_\delta}\right] \le C_{T,\delta,p}.
$$
The proof is complete.
\end{proof}

By Lemma  \ref{l:ZEst} and Lemma  \ref{l:Y2},  we  obtain that
\begin{lem}\label{l: X}
For all $T>0$,   $\delta \in (0,1/2)$ and $p \in (0,\alpha/4)$, we have
$$
\E_{x}\left[\|X_{T}\|^p_{H_\delta}\right] \le C_{T,\delta,p}
$$
where the constant $C_{T,\delta,p}$   does not depend on the initial  value $X_0=x$.
\end{lem}

\subsection{Hyper-exponential Recurrence}

In this subsection, we will verify the hyper-exponential recurrence condition \eqref{condition 2}.
\vskip0.3cm

For any $\delta\in(0,1/2), M>0$, define the hitting time  of $\{X_n\}_{n\ge1}$:
\begin{equation}
\tau_M=\inf\{k\ge1;\|X_{k}\|_{H_{\delta}}\le M\}.
\end{equation}
 Let
 $$
 K:=\{x\in H_{\delta}; \|x\|_{H_{\delta}}\le M\}.
 $$
 Clearly, $K$ is compact in $H$. Recall the definitions of  $\tau_K$ and $\tau_K^1$ in \eqref{stopping time}. It is obvious  that
\begin{equation}
\tau_K\le \tau_M,\ \ \ \ \ \tau_K^{1}\le \tau_M.
\end{equation}
This fact, together with the following important theorem,  implies the hyper-exponential recurrence condition  \eqref{condition 2}.

\begin{theorem}\label{thm hyper-exp} For any $\lambda>0$,  there exists $M=M_{\lambda, \delta}$  such that
$$
\sup_{\nu\in\mathcal M_1(H)}\E^{\nu}[e^{\lambda\tau_M}]<\infty.
$$
\end{theorem}
\begin{proof}
For any $n\in\N$, let
$$
B_n:=\left\{\|X_{j}\|_{H_{\delta}}>M; j=1,\cdots,n\right\}=\{\tau_M>n\}.
$$
By the Markov property of $\{X_n \}_{n\in\N}$, Chebychev's inequality and Lemma \ref{l: X}, we obtain that for any $\nu\in \M_1(H)$, $p \in (0,\alpha/4)$,
\begin{align*}
\PP_{\nu}(B_n)=&\PP_{\nu}(B_{n-1})\cdot\PP_{\nu}(B_n|B_{n-1})\\
\le&\PP_{\nu}(B_{n-1})\cdot \frac{\E_{X_{n-1}}\left[\|X_n\|_{H_{\delta}}^p\right]}{M^p}\\
\le & \PP_{\nu}(B_{n-1})\cdot\frac{C_{\delta,p}}{M^p},
\end{align*}
where $C_{\delta,p}$ is the constant in Lemma \ref{l: X} (taking $T=1$).

By induction, we have for any $n\ge0$,
$$
\PP_{\nu}(\tau_M>n)=\PP_{\nu}(B_n)\le \left(\frac{C_{\delta,p}}{M^p}\right)^n.
$$
This inequality, together with Fubini's theorem, implies that for any $\lambda>0,\nu\in\M_1(H)$,
\begin{align*}
\E_{\nu}\left[ e^{\lambda\tau_M}\right]=&\int_0^{\infty}\lambda e^{\lambda t}\PP_{\nu}(\tau_M>t)\dif t\\
\le &\sum_{n=0}^{\infty}\lambda e^{\lambda (n+1)}\PP_{\nu}(\tau_M>n)\\
\le &\sum_{n=0}^{\infty}\lambda e^{\lambda (n+1)}\left(\frac{C_{\delta,p}}{M^p}\right)^n,
\end{align*}
which is finite as $M>(C_{\delta,p} e^{\lambda})^{1/p}$.

The proof is complete.
\end{proof}

\vskip0.3cm

\noindent{\bf Acknowledgments}: We would like to gratefully thank Armen Shirikyan for pointing out the key estimate \eqref{e:YtEstZt} for us. R. Wang thanks the  Faculty of Science and Technology, University of Macau, for finance support and hospitality.    He was supported by Natural Science Foundation of China 11301498, 11431014 and the Fundamental Research Funds for the Central Universities WK0010000048. J. Xiong was  supported by Macao Science and Technology Fund FDCT 076/2012/A3 and Multi-Year Research Grants of the University of Macau Nos. MYRG2014-00015-FST and MYRG2014-00034-FST. L. Xu is supported by the grants: SRG2013-00064-FST, MYRG2015-00021-FST and Science and Technology Development Fund, Macao S.A.R FDCT 049/2014/A1. All of the three authors are supported by the research project RDAO/RTO/0386-SF/2014.

%%%%%%%%%%%%%%%%%%%%%%%%%%%%%%%%%%%%%%%%%%

\bibliographystyle{amsplain}

\end{document}